\documentclass[a4paper,12pt]{article}

\usepackage[truedimen,margin=20truemm]{geometry}

\usepackage{amsmath,amsthm,comment,amssymb}
\usepackage{graphicx}
\usepackage{color}
\newcommand{\Prob}{\mathrm{Prob}}

\newcommand{\rE}{\mathrm{E}}

\newcommand{\VaR}{\mathrm{VaR}}

\newcommand{\cA}{\mathcal{A}}

\newcommand{\cP}{\mathcal{P}}

\newcommand{\cC}{\mathcal{C}}

\newcommand{\cN}{\mathcal{N}}

\newcommand{\cX}{\mathcal{X}}
\newcommand{\cY}{\mathcal{Y}}
\newcommand{\cZ}{\mathcal{Z}}

\newcommand{\bR}{\mathbf{R}}

\usepackage{algorithm}
\usepackage{algorithmic}
\usepackage{enumitem}

\usepackage[T1]{fontenc}
\usepackage[utf8]{inputenc}
\usepackage[font=small,labelfont=bf,tableposition=top]{caption}

\theoremstyle{plain}
\newtheorem{thm}{Theorem}
\newtheorem{lem}[thm]{Lemma}
\newtheorem{prop}[thm]{Proposition}

\newtheorem{fact}[thm]{Fact}

\theoremstyle{definition}
\newtheorem{df}[thm]{Definition}
\newtheorem{ex}[thm]{Example}
\newtheorem{prob}[thm]{Problem}
\newtheorem{assump}[thm]{Assumption}

\newtheorem{rmk}[thm]{Remark}

\usepackage[utf8]{inputenc} 
\usepackage[T1]{fontenc}    
\usepackage{hyperref}       
\usepackage{url}            
\usepackage{booktabs}       
\usepackage{amsfonts}       
\usepackage{nicefrac}       
\usepackage{microtype}      

\title{Tightly Robust Optimization via Empirical Domain Reduction}

\author{%
  Akihiro Yabe\\
  NEC corporation\\
  \texttt{a\_yabe@nec.com} \\
  \and
  Takanori Maehara\\
  RIKEN AIP\\
  \texttt{takanori.maehara@riken.jp} \\
}

\date{}

\begin{document}

\maketitle

\begin{abstract}
	Data-driven decision-making is performed by solving a parameterized optimization problem, and the optimal decision is given by an optimal solution for unknown true parameters.
	We often need a solution that satisfies true constraints even though these are unknown.
	Robust optimization is employed to obtain such a solution, where the uncertainty of the parameter is represented by an ellipsoid, and the scale of robustness is controlled by a coefficient.
	In this study, we propose an algorithm to determine the scale such that the solution has a good objective value and satisfies the true constraints with a given confidence probability.
	Under some regularity conditions, the scale obtained by our algorithm is asymptotically $O(1/\sqrt{n})$, whereas the scale obtained by a standard approach is $O(\sqrt{d/n})$.
	This means that our algorithm is less affected by the dimensionality of the parameters.
\end{abstract}

\section{Introduction}\label{secIntro}
\subsection{Problem Setting}
Data-driven decision-making must manage stochastic uncertainty due to a limited number of available data samples.
Suppose that a decision is made by solving the following parameterized optimization problem~\cite{bertsimas2011theory}:
\begin{align}
	\begin{array}{ll}
		& \displaystyle \min_{x \in \cX} f(x) \\
		& \text{s. t. } g_k (x, \theta) \geq 0, \quad k=1,\dots,K
	\end{array} \label{nominalOpt} 
\end{align}
where $\theta \in \bR^d$ is a parameter, $\cX \subseteq \bR^m$ is a decision domain, $x \in \cX$ is a decision variable, $f: \bR^m \to \bR$ is an objective function, and $g_1, \dots, g_K: \bR^m \times \bR^d \to \bR$ are parameterized constraint functions that are linear in $\theta$.
The ideal decision is given by an optimal solution $x(\theta^*)$ under true parameter $\theta^*$;
however, $\theta^*$ is unknown to us and must be estimated from i.i.d. samples $\theta^i$ ($i = 1, \dots, n$) from a distribution $P^*$ whose expectation is $\theta^*$.

In practice, we often need a solution that satisfies the true constraints (i.e., constraints with true parameters) even though these are unknown.
In such a case, a naive approach that solves the problem with estimated parameters would be inadequate because the obtained solution might not satisfy the true constraints.
Rather, robust optimization~\cite{ben2009robust,bertsimas2011theory} can be used as follows.
For a positive semidefinite matrix $\Sigma$, we define $U_{\Sigma} := \{ u \in \bR^d \mid u^{\top} \Sigma^{-1} u \leq 1 \}$.
The \emph{robust optimization problem (with ellipsoidal uncertainty)} is then given by
\begin{align}
	\begin{array}{ll}
		& \displaystyle \min_{x \in \cX} f(x)  \\
		& \displaystyle \text{s.t. } \min_{u \in U_\Sigma} g_k (x, \theta + \lambda u) \geq 0, \quad k=1,\dots,K  
	\end{array}
	\label{robustOpt}
\end{align}
where $\lambda$ is a parameter called the nonnegative scale.
Let $x(\lambda ; \theta, \Sigma)$ be the solution of the above problem.
Then, $x(\lambda ; \theta, \Sigma)$ satisfies the true constraints if $\theta^* - \theta \in \lambda U_\Sigma$.
This means that if $\lambda$ is large, it will most likely satisfy the true constraints but may have a poor best-case performance.
Conversely, if $\lambda$ is small, it will have a good best-case performance but also a high risk of violating the true constraints.
Therefore, the choice of $\lambda$ is a significantly important task~\cite{bertsimas2004price}.

In this study, we consider the problem of finding a $\lambda$ such that the problem will have a small objective value and satisfy the true constraints with a sufficiently high probability.
Formally, the problem is defined as follows. 
Let $f^*$ be the true objective function:
\begin{align*}
	f^*(x) :=
	\begin{cases}
		f(x) & \text{ if } g_k (x, \theta^*) \geq 0, \quad k=1,\dots,K,\\
		\infty & \text{otherwise}.
	\end{cases}
\end{align*}
Let $\hat{\theta}_n$ be the empirical average of $\theta^1, \dots, \theta^n$,
$\hat{\Sigma}_n$ be an estimate of the covariance matrix $\Sigma^*$ defined by $\hat{\Sigma}_n := (1/ n) \sum_{i=1}^n (\theta^i  - \hat{\theta}_n)(\theta^i -  \hat{\theta}_n )^{\top}$, 
and $\hat{x}_n(\lambda)$ be defined by $\hat{x}_n(\lambda) := x(\lambda; \hat{\theta}_n, \hat{\Sigma}_n)$.
Our goal then is to find an algorithm $\mathcal{A}$ that determines a scale $\lambda$ from the observation $\theta^n$ such that the ($1 - \delta$)-quantile of $f^*(\hat{x}(\cA(\theta^n)))$, i.e.,  the \emph{value-at-risk} $\VaR_\delta[ f^*(\hat{x}(\cA(\theta^n))) ]$, is minimized as follows~\footnote{Here we define $\VaR_{\delta} [f^*(\hat{x}(\cA(\theta^n))) ] := \inf \{\eta \mid  \Prob(f^*(\hat{x}(\cA(\theta^n))) \leq \eta) \geq 1- \delta \} $, where the probability is taken for the i.i.d. sampling of $\theta^n$.}:
\begin{prob}\label{prob}
	Given $\delta > 0$,
	find an algorithm $\cA: \theta^n \mapsto \lambda$ 
	that minimizes $\VaR_{\delta}[f^*_n(\hat{x}(\cA(\theta^n))) ]$.
\end{prob}
\subsection{Standard Approach}\label{subsecBaseline}
%

A standard approach for Problem~\ref{prob} determines a $\lambda$ such that $\hat{\theta}_n - \theta^* \in \lambda U_{\Sigma}$ holds with a desired probability~\cite{ben2013robust,bertsimas2018data,delage2010distributionally}.
Let $\Sigma^* := E_{\theta \sim \cP} [(\theta - \theta^*)(\theta - \theta^*)^{\top}] $ be the true covariance matrix and $x^*(\lambda) := x(\lambda; \theta^*, \Sigma^*)$.
We then obtain the following:
\begin{prop}\label{propBase}
	Suppose that $\hat{\theta}_n - \theta^* \in \lambda U_{\Sigma^*}$
	with probability at least $1-\delta$.
	It then holds that $\VaR_{\delta}[f^*(x(\lambda ; \hat{\theta}, \Sigma^*) )] \leq f^*(x^*(2 \lambda ))$.
\end{prop}

If $\hat{\theta}_n$ is (asymptotically) distributed by a normal distribution $\cN(\theta^*,\Sigma^*/n)$ with mean $\theta^*$ and covariance $\Sigma^*/n$,
then such $\lambda$ will be determined by $\lambda = \chi_d^{-1}(1 - \delta) / \sqrt{n}$,
where $\chi_d$ is a chi-distribution with a degree of freedom $d$.

Observe that such a $\lambda$ will be of $O(\sqrt{d/n} )$.
Thus, the speed of convergence depends on the dimension $d$ of the parameter $\theta$.
The speed of convergence will be especially degraded 
when a large number of features ($d \gg 1$) are available, but only a small number of features is useful,
which will cause the so-called ``curse of dimensionality'' in robust optimization.

\subsection{Our Approach and Results}\label{subsecApproach}
In this study, we prove the following theorem, which improves the convergence rate of Corollary~\ref{propBase}:
\begin{thm}[Informal version of Theorem~\ref{thmFormal}] \label{thmInformal}
	There is an algorithm $\cA$, which is given in Algorithm~\ref{algoReduction}, such that the output $\hat{\lambda} = \cA(\theta^n)$ satisfies
	\begin{align*}
		\VaR_{\delta}[f^*(\hat{x}_n(\hat{\lambda})) ] \leq f^*(x^*(\overline{\lambda}_n)),
	\end{align*}
	where 
	\begin{align*}
		\lim_{n \to \infty} \sqrt{n} \overline{\lambda}_n \leq \chi_1^{-1}(1 - \delta/K).
	\end{align*}
\end{thm}
It is important to note that our $\overline{\lambda}_n$ is asymptotically independent of the dimensionality of the parameter $\theta$,
i.e., the speed of convergence is less affected by the dimensionality of the parameter $\theta$.

The theorem is proved in the following steps. 
We first derive a new inequality on a subspace.
We then establish methods to apply the new inequality effectively.

\subsubsection{New Inequality on Subspace}
Recall that Corollary~\ref{propBase} is derived from the following fundamental fact in robust optimization.
\begin{fact}\label{factConf}
	Let $\lambda \geq 0$. If $\theta \in \bR^d$ and $\Sigma \succeq O$ satisfy $\theta - \theta^* \in \lambda U_{\Sigma}$,
	then 
	\begin{align}
		f^*(x(\lambda; \theta, \Sigma) ) \leq f^*(x(2\lambda; \theta^*, \Sigma)). \label{ineqFact}
	\end{align}
\end{fact}
This fact holds because
the assumption $\theta - \theta^* \in \lambda U_{\Sigma}$ implies
$\min_{u \in U_\Sigma} g (x, \theta + \lambda u) \leq g (x, \theta^*)$ for \emph{arbitrary $g: \bR^m \times \bR^d \to \bR$ and $x \in \bR^m$}. 
Here, we observe that this is too conservative.
If we consider given $g_k$ ($k=1,\dots,K$) and a subspace $\cY \subseteq \cX$ that includes the resulting solution $x(\lambda; \theta, \Sigma)$, 
we can prove the same inequality~\eqref{ineqFact} under a weaker assumption.
Let $r_k(x; \Sigma) := \max_{u \in U_{\Sigma}} g_k(x, u)$.
For $\cY \subseteq \bR^d$, let $S(\lambda, \cY; \Sigma)$ be the set of errors that does not cause a violation of the true constraint over $\cY$, i.e.,
\begin{align*}
	S(\lambda, \cY; \Sigma) := \left\{\Delta \in \bR^d \left| 
	\begin{array}{l}
		|g_k (x, \Delta)| \leq \lambda r_k(y; \Sigma), \\
		\quad k=1,\dots,K, \forall y \in \cY
	\end{array}
	\right.  \right\}. 
\end{align*}
We then obtain the following lemma: 
\begin{lem}\label{lemConf}
	Let $\lambda \geq 0$. If $\theta \in \bR^d$, $\Sigma \succeq O$, and $\cY \subseteq \bR^d$ satisfy
	\begin{align}
		\theta - \theta^* \in S(\lambda, \cY; \Sigma), \label{sufCond1} \\
		x(\lambda; \theta, \Sigma), x(2\lambda; \theta^*, \Sigma) \in \cY, \label{sufCond2}
	\end{align}
	then \eqref{ineqFact} holds.
\end{lem}
Lemma~\ref{lemConf} generalizes Fact~\ref{factConf} because the assumption of Fact~\ref{factConf} is a sufficient condition in Lemma~\ref{lemConf} for $\cY = \cX$ since $\lambda U_{\Sigma} \subseteq S(\lambda, \cY; \Sigma) $ holds for any $\cY$, and \eqref{sufCond2} is trivial with $\cY = \cX$.
Our idea is to use Lemma~\ref{lemConf} rather than Fact~\ref{factConf} to derive better bounds than Corollary~\ref{propBase}.

The following example shows the effect of choosing a different $\cY$.
\begin{ex}
	Let us consider the following $d = 2$-dimensional parameterized problem:
	\begin{align}
		\begin{array}{ll}
			&\displaystyle \min_{x_1, x_2 \geq 0} x_1 + x_2 \\
			&\text{s.t. } \theta_1 x_1 + \theta_2 x_2 - 2 \geq 0.
		\end{array}
	\end{align}
	The true parameters are $\theta^* = (2,1)^{\top}$, and the true covariance of samples is $\Sigma^* = I_2$, where $I_d$ is an identity matrix of size $d$.
	
	Figure~\ref{figToyEx} illustrates $S(\lambda,\cY) = S(\lambda,\cY; I_2)$ for 
	$\cY = \bR^2$, $\bR_+^2$, and $\bR^{x_1}_+ := \{(x_1,0)^{\top} \in \bR^2  \mid x_1 \geq 0\}$, where all these $\cY$ include the true optimum solution $x^* = (1,0)^{\top}$.
	The figure clearly shows that, with a fixed $\lambda$, a smaller $\cY$ results in a larger $S$:
	\begin{align*}
		\lambda U_{\Sigma^*} = S(\lambda, \bR^{2}) \subsetneq S(\lambda, \bR^{2}_+) \subsetneq S(\lambda, \bR^{x_1}_+).
	\end{align*}
	Recall that Fact~\ref{factConf} and Lemma~\ref{lemConf} lead us to define $\lambda$ in order to satisfy $\hat{\theta}_n - \theta^* \in S(\lambda, \cY)$ with desired probability $1-\delta$.
	
	More specifically, let $\hat{\theta}_n$ be distributed by the normal distribution $\cN(\theta^*, I_2)$ and let $\delta = 0.1$.
	Fact~\ref{factConf} with $U_{I_2} (=  S(\lambda, \bR^{2}))$ then requires
	$\lambda = \chi_2^{-1}(1-\delta) \approx 2.1$,
	while Lemma~\ref{lemConf} together with $S(\lambda, \bR^{x_1}_+)$ requires $\lambda = \chi_1^{-1}(1-\delta) \approx 1.6$.
	Since a smaller $\lambda$ yields a closer solution to the optimal one, this indicates the possibility of having a better solution.
	\begin{figure}[t]
		\centering
		\includegraphics[width=0.5\hsize]{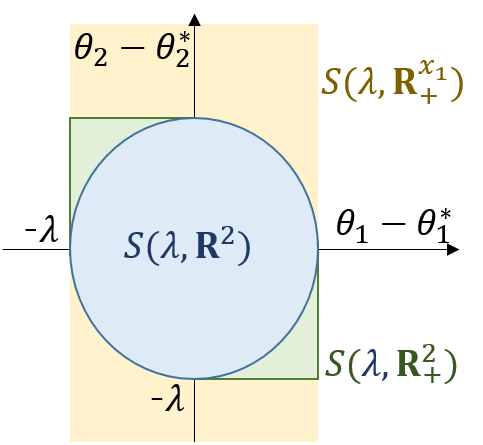}
		\caption{Plot of $S(\lambda,\cY) = S(\lambda,\cY; I_2)$ for $\cY = \bR^2$, $\bR_+^2$, and $\bR^{x_1}_+$, in a two-dimensional toy problem.} \label{figToyEx}
	\end{figure}
	\qed
\end{ex}

For an effective use of Lemma~\ref{lemConf},
we must be able to evaluate the probability $\Prob(\hat{\theta}_n - \theta^* \in S(\lambda, \cY; \Sigma^*))$ for $\cY \subseteq \cX$,
and must find a good subspace $\cY \subseteq \cX$
without knowing $\theta^*$ and $\cP^*$.

\subsubsection{Probability Evaluation}

In order to evaluate probability without $\theta^*$ and $\cP^*$, 
in Section~\ref{secSample} we characterize $S(\lambda, \cY ; \Sigma^*)$ as an intersection of Minkovski's sum of an ellipsoid and a dual cone.
Since such an intersection is a convex set,
we can apply the high-dimensional Berry-Esseen theorem for convex sets~\cite{bentkus2003dependence}
in order to approximate the distribution of $\hat{\theta}_n - \theta^*$ by a normal distribution
while maintaining a theoretical guarantee.
We then propose a sampling-based algorithm for measuring the probability with sufficient accuracy.

\subsubsection{Subspace Selection}

To find good subspace $\cY \subseteq \cX$,
in Section~\ref{secReduction} we extend the two-step empirical domain reduction algorithm proposed for risk minimization with a single uncertainty objective~\cite{yabe2019empirical}
to our context for robust optimization with multiple uncertain constraints.
Such a reduced space must satisfy \eqref{sufCond1} and \eqref{sufCond2} in Lemma~\ref{lemConf}
with high probability.
For \eqref{sufCond1}, since the reduction is based on empirical estimate $\hat{\theta}_n$,
the resulting $S(\lambda, \cY ; \Sigma)$ also includes stochastic uncertainty.
We must thus avoid over-fitting on a given single sample and keep consistency with other sampling scenarios.
In addition, \eqref{sufCond2} requires that the solution $x(\lambda; \theta, \Sigma)$ must be included in $\cY$ without forcing it by adding constraints.
These two requirements require precise control of domain reduction,
which can be done using our empirical two-step domain reduction algorithm.

\subsection{Related studies}
Our approach can be summarized in terms of the following concepts: inequality on subspaces,
probability evaluation, and subspace selection.
While these concepts are novel in the context of robust optimization,
they have already been proven to be effective 
in the context of variance-based regularization for risk minimization in machine learning.
Let us note that variance-based regularization regularizes uncertain objectives
by means of scaled (square-root of) variance,
and thus it is similar to the robust optimization with ellipsoidal uncertainty
that introduces redundancy into uncertain constraints, which redundancy is proportional to (square-root of) variance.

In the context of risk minimization,
the relationship between the size of an optimization domain and the speed of convergence 
has been characterized by various complexity measures,
such as VC dimension~\cite{vapnik1971uniform}, covering number~\cite{maurer2009empirical},
and Rademacher complexity~\cite{bartlett2002rademacher}.
Studies of variance-based regularization~\cite{maurer2009empirical,namkoong2017variance}
have determined the scale of a regularizer on the basis of these complexity measures
in order for a regularized empirical risk function to bound the true objectives with a desired probabilities.
A recent study of empirical hypothesis space reduction~\cite{yabe2019empirical}
has achieved, by means of subspace selection, an acceleration of convergence that is asymptotically independent of the dimensionality of uncertain parameters.

To make use of the approaches in~\cite{yabe2019empirical}, which is designed for variance-based regularization,
in robust optimization, we have dealt with the issues below.
Robust optimization has several uncertain constraints,
while risk minimization has a single uncertain objective.
In particular, while an uncertainty objective does not influence the feasible domain,
uncertain constraints do influence,
which increases the difficulty of subspace selection.
In addition, in robust optimization, it is often assumed that uncertain functions are linear with respect to uncertain parameters, and
this linearity has led us to propose a novel sampling-based evaluation of violation probability.

\section{Preliminary}\label{secPre}
\subsection{Technical lemmas}
This section introduces a series of lemmas.
The following Gershgorin circle theorem characterizes the list of eigenvalues of a matrix $A \in \bR^{d \times d}$
by its elements $a_{ij}$.
\begin{lem}[The Gershgorin circle theorem; see~\cite{golub1989matriz}]\label{lemGersh}
	Let $A = (a_{ij}) \in \bR^{d \times d}$, and define $ R_i = \sum_{j\neq i} |a_{ij}|$.
	Under suitable ordering, the set of eigenvalues $\lambda_1,\dots,\lambda_d$ of $A$ satisfy
	$ a_{ii} - R_i \leq \lambda_i \leq a_{ii} + R_i$.
\end{lem}
Hoeffding's inequality below bounds the gap between sample average and true average.
\begin{lem}[Hoeffding's inequality; see~\cite{boucheron2013concentration}]\label{lemHoeffding}
	Let $Z,Z_1,\dots,Z_n$ be i.i.d. random variables with values in $[0,1]$ and let $\delta>0$.
	Then we have
	\begin{align*}
		\Prob \left(  \rE[Z] - \frac{1}{n}\sum_{i=1}^n Z_i \leq \sqrt{\frac{\log1/ \delta}{2n}} \right) \geq 1 - \delta, \\
		\Prob \left(\frac{1}{n}\sum_{i=1}^n Z_i  -  \rE[Z] \leq \sqrt{\frac{\log 1/\delta}{2n}} \right) \geq 1 - \delta.
	\end{align*}
\end{lem}
The following Berry-Esseen theorem enables us to evaluate the speed of convergence of the central limit theorem (for tighter coefficient $c_1$ in a one-dimensional case, see~\cite{korolev2010upper}).
\begin{lem}[The high-dimensional Berry-Esseen inequality~\cite{bentkus2003dependence}] \label{lemHighBE}
	Let $d \ge 1$ be an integer.
	Let $\cC_d$ be a set of convex set in $\bR^d$ and $c_d := 400 d^{1/4}$.
	Let $Z_1,\dots, Z_n$ be i.i.d. random variables over $\bR^d$ with $E[Z_i] = 0$, $E[Z_i Z_i^{\top}] = I_d$, and $E[\|Z_i \|^3]  \leq \tau$.
	Let $Z = (1/\sqrt{n}) \sum_{i=1}^n Z_i$.
	It then holds that 
	\begin{align*}
		\sup_{C \in \cC_d} \left| \underset{W \sim \cN(0,I_d)}{\Prob} (W \in C) - \underset{Z}{\Prob} (Z \in C) \right| \leq \frac{c_d \tau }{\sqrt{n}}
	\end{align*}
\end{lem}

\subsection{Estimation accuracy}
We can show here the accuracy of estimators $\hat{\theta}_n$ and $\hat{\Sigma}_n$ on the basis of the lemmas introduced 
in the previous section and the following assumption. Let $\cP^*_{\Delta}$ denote the distribution of $\Delta = \Sigma^{* -1/2}(\theta - \theta^*)$
where $\theta \sim \cP^*$.
\begin{assump}\label{assump}
	An upper-bound $\tau$ of  $ E_{\Delta \sim \cP^*_{\Delta}} [\| \Delta \|^3]$ and $\tau'$ of
	$\max_{j=1,2,\dots,d} E_{\Delta \sim \cP^*_{\Delta}} [ \Delta_j^6]$ is given.
\end{assump}
We here assume that (an upper-bound of) the third moment $\tau$ and sixth moment $\tau'$ is known,
since the speed of convergence is less influenced by the higher-order moments $\tau$ and $\tau'$ than by the first and second moments ($\theta^*$ and $\Sigma^*$).
In practice, $\tau$ and $\tau'$ can be calculated if the domain of the distribution $\cP^*$ is bounded,
or it can also be estimated using samples $\theta^n$.

Under Assumption~\ref{assump}, the distribution of $\hat{\theta}_n$ converges to the normal distribution $\cN(0,\Sigma^* / n)$ by the central limit theorem.
Recall the definitions of $\cC_d$ and $c_d := 400 d^{1/4}$ in Lemma~\ref{lemHighBE}.
The speed of convergence can be characterized by $\tau$, $\tau'$, and Lemma~\ref{lemHighBE} as follows:
\begin{lem} \label{lemParaEst}
	It holds that
	\begin{align*}
		\sup_{C \in \cC_d} \left| \underset{\Delta \sim \cN(0, \Sigma^*/n)}{\Prob} (\Delta \in C) - \underset{\hat{\theta}_n}{\Prob} (\hat{\theta}_n - \theta^* \in C) \right| \leq \frac{c_d \tau}{\sqrt{n}}.
	\end{align*}
\end{lem}
For $\delta' >0$, let us define $\eta_{d,n}(\delta')$ by
\begin{align*}
	\eta_{d,n}(\delta') := \frac{\tau^{'1/3} d}{\sqrt{n}} \chi_1^{-1}\left( 1 - \frac{\delta'}{d^2}  + \frac{c_1 \tau'}{ \sqrt{n}}  \right).
\end{align*}
The relative accuracy of estimator $\hat{\Sigma}_n$ can be characterized by Lemma~\ref{lemGersh} and Lemma~\ref{lemHighBE} as follows.
\begin{lem}\label{lemCovEst}
	For $\delta' >0$, the following then holds with probability at least $ 1 - \delta'$:
	\begin{align}
		\left(1-  \eta_{d,n}(\delta') \right)
		\Sigma^* 
		\preceq \hat{\Sigma}_n
		\preceq \left( 1+ \eta_{d,n}(\delta') \right)\Sigma^*. \label{approx2}
	\end{align}
\end{lem}
Note that the above bound can be improved by using the Berry-Esseen theorem that is specialized for a one-dimensional random variable, rather than by using Lemma~\ref{lemHighBE}.
Since this improvement does not influence the order of convergence,
we here use Lemma~\ref{lemHighBE} for simplicity of presentation.

\section{Probability Evaluation Algorithm}\label{secSample}
This section proposes an algorithm that realizes the concept of probability evaluation
discussed in Section~\ref{subsecApproach}.
\subsection{Geometric characterization of $S(\mu, \cY; \Sigma)$}
Recall that $S(\lambda, \cY; \Sigma) \subseteq \bR^d$ is the set of estimation error $\theta - \theta^*$
that does not cause a violation of the true constraint over $\cY$ with the scale $\lambda$. 
Here we prove that $S$ is a convex set for any $\lambda$, $\cY$, and $\Sigma$.

Since $g_k$ is linear in $\theta$, there exists a function $v_k : \cX \to \bR^d$ such that
\begin{align*}
	g_k(x, \theta) = \theta^{\top} v_k(x).
\end{align*}
For $\cY \subseteq \cX$, we define $V_k(\cY)$ by
\begin{align*}
	V_k(\cY) := \{ v_k (y) \in \bR^d  \mid y \in \cY \}. 
\end{align*}
We then define the polar cone $C_{k}^+(\cY)$ and dual cone $C_{k}^-(\cY)$ of $V_k(\cY)$ by
\begin{align*}
	C_{k}^+(\cY) := \{ \Delta \in \bR^d \mid \Delta^{\top} v \geq 0, \forall v \in V_k(\cY) \},\\
	C_{k}^-(\cY) := \{ \Delta \in \bR^d \mid \Delta^{\top} v \leq 0, \forall v \in V_k(\cY) \} .
\end{align*}
The polar cone $C_{k}^+(\cY)$ and the dual cone $C_{k}^-(\cY)$ are convex cones.
For a pair of set $U, S \subseteq \bR^d$, their Minkowski's sum is defined by $U + S := \{u+s \mid u \in U, s \in S \}$.
Note that Minkowski's sum of two convex sets is also a convex set.
The following theorem characterize $S(\mu, \cY; \Sigma)$ as an intersection of convex sets:
\begin{lem}\label{lemSConvex}
	For any $\lambda \geq 0$, $\cY \subseteq \cX$, and $\Sigma \succeq 0$, $S(\lambda, \cY; \Sigma)$ is a convex set.
	In particular, if $g_k(x,\theta)$ is linear in $x$ for all $k=1,2,\dots,K$ and $\cY$ is convex,
	then it holds that
	\begin{align*}
		S(\lambda, \cY; \Sigma) 
		:= \left(\bigcap_{k=1}^K \lambda U_{\Sigma} + C_{k}^+ (\cY) \right) \cap \left( \bigcap_{k=1}^K \lambda U_{\Sigma^*} + C_{k}^- (\cY) \right) .
	\end{align*}
\end{lem}
\subsection{Normal approximation via multidimensional Berry-Esseen' theorem}
On the basis of~\cite[Definition 5]{yabe2019empirical},
we here introduce the minimum spatial uniform bounds $\mu (\alpha, \cY; \cP, \Sigma)$
as a minimum scale $\mu \geq 0$ that satisfies \eqref{sufCond1} with probability $\alpha$
when $\theta - \theta^*$ is distributed by $\cP$:
\begin{df}\label{dfMu}
	Let $p \in [0,1]$, $\cY \subseteq \cX$, $\Sigma \succeq O $, and $\cP$ be a distribution over $\bR^d$.
	The minimum spatial uniform bound $\mu (\alpha, \cY; \cP, \Sigma) \in \bR_+$
	is defined by:
	\begin{align}
		\mu(\alpha, \cY; \cP, \Sigma)
		:= \inf \left\{ \mu \geq 0 \mid  \Prob_{\theta \sim \cP} \left(\theta \in  S(\mu, \cY; \Sigma) \right)   \geq p \right\}.  \label{eqDefMu} 
	\end{align}
\end{df}
Let us define $\cP^*_n$ as the distribution of $\hat{\theta}_n - \theta^*$.
The ideal goal of this section, then, is to obtain $\mu(\alpha, \cY; \cP^*_n, \Sigma^*)$.
$\mu$ satisfies the following component-wise monotonicity.
Let us denote $\cN_{\Sigma} := \cN(0, \Sigma)$.
%
\begin{lem}\label{lemMonotonicity}
	Let $p \in [0,1]$, $\cY \subseteq \cX$, $\Sigma \succeq O $, and $\cP$.
	Then the following holds.
	
	(i) $\mu(p, \cY; \cP, \Sigma) \leq \mu(p', \cY; \cP, \Sigma)$ for all $p \leq p'$.
	[Monotonicity in $p$]
	
	(ii)  $\mu(p, \cY; \cP, \Sigma) \geq \mu(p, \cZ; \cP, \Sigma)$ for all $\cY \subseteq \cZ$.
	[Reverse Monotonicity in $\cY$]
	
	(iii) $\mu(p, \cY; \cN_{\Sigma_1}, \Sigma) \geq \mu(p, \cY; \cN_{\Sigma_2}, \Sigma)$ and $\mu(p, \cY; \cP, \Sigma_1) \geq \mu(p, \cY; \cN_{\Sigma_2}, \Sigma_2)$ for all $\Sigma_1 \preceq \Sigma_2$.
	[Reverse Monotonicity in $\Sigma$]
	
	(iv) $\mu(p, \cY; \cN_{\nu_1 \Sigma}, \nu_2 \Sigma) = \sqrt{\nu_1 / \nu_2} \mu(p, \cY; \cN_{\Sigma}, \Sigma)$ for all $\nu_1, \nu_2 > 0$.
	[Scaling]
\end{lem}
Recall that Lemma~\ref{lemSConvex} shows the convexity of $S(\mu, \cY; \Sigma)$ for arbitrary $\mu$, $\cY$, and $\Sigma$.
We can then apply Lemma~\ref{lemParaEst} to bound the gap of probability in \eqref{eqDefMu}
for $\cP^*_n$ and the corresponding normal $\cN_{\Sigma^* / n}$.
Combined with the monotonicity shown in Lemma~\ref{lemMonotonicity} (i),
we can upper-bound $\mu(\alpha, \cY; \cP^*_n, \Sigma^*)$ of unknown distribution $\cP^*_n$
by its asymptotically normal counterpart:
\begin{prop}\label{propNormalApprox}
	It holds that
	\begin{align*}
		\frac{1}{\sqrt{n}} \mu \left( \alpha  - \frac{c_d \tau}{\sqrt{n}} , \cY; \cN_{\Sigma^*}, \Sigma^* \right) 
		\leq \mu(\alpha, \cY; \cP^*_n, \Sigma^*) 
		\leq  \frac{1}{\sqrt{n}} \mu \left( \alpha  + \frac{c_d \tau}{\sqrt{n}} , \cY; \cN_{\Sigma^*}, \Sigma^* \right).
	\end{align*}
\end{prop}
\subsection{Sampling algorithm for calculating spatial uniform bounds for normal distribution}\label{subsecSampleAlgo}
This section fixes $p \in [0,1]$, $\cY \subseteq \cX$, and $\Sigma \succeq O$,
and we then denote $\mu'(p) = \mu(p, \cY ; \cN_{\Sigma}, \Sigma)$ for notational simplicity.
Proposition~\ref{propNormalApprox} reduces our goal to that of calculating a minimum spatial uniform bound $\mu'(p)$.
Lemma~\ref{lemMonotonicity} (i) shows that such a $\mu'$ is monotone in $p$,
and $\mu'$ satisfies the following property:
\begin{lem}\label{lemAlgo}
	(i) It holds that $\chi_1^{-1}(p) \leq \mu'(p) \leq \chi_d^{-1}(p)$.
	
	(ii) It holds that 
	\begin{align*}
		&\mu^{' -1}(\lambda)  := \underset{ \Delta \sim \cN_{\Sigma} }{\Prob} 
		\left(
		\begin{array}{r}
			\min_{y \in \cY} \lambda r_k(y; \Sigma) + g_k(y, \Delta) \geq 0,\\
			\min_{y \in \cY} \lambda r_k(y; \Sigma) - g_k(y, \Delta) \geq 0,\\
			\forall k =1,2,\dots,K
		\end{array}\right) .
	\end{align*}
\end{lem}
Thus $\mu'$ is monotone, upper and lower bounded,
and its inverse $\mu^{'-1}$ can be calculated by a series of sampling and optimization.
This characterization leads the following algorithm for calculating $\mu'$ by sampling from $\cN_{\Sigma}$ and a binary search, as shown in Algorithm~\ref{algoSample}.

Given $p \in [0,1]$, $\cY \subseteq \cX$, and $\Sigma \succeq O $, together with parameters $\alpha, \beta, \gamma \geq 0$ deciding accuracy,
Algorithm~\ref{algoSample} calculates a $\lambda$ that approximates $\mu'(p)$.
Line~\ref{algoDefQ} of the algorithm first defines the number $Q$ of samples by
\begin{align}
	Q := \left\lceil \frac{\log(2 / \alpha)}{8 \beta^2} \right\rceil. \label{defNumSample}
\end{align}
Line~\ref{algoSampling}, then, generates samples $\varepsilon_i$ for $i=1,2,\dots,Q$ from the normal distribution $\cN_{\Sigma}$.
Lines~\ref{algoBinSearchStart}--\ref{algoBinSearchEnd}
estimate $\mu' (p)$ by the binary search.
Line~\ref{algoBinSearchInit} defines the upper-bound $\overline{\lambda}$ and the lower-bound $\underline{\lambda}$
on the basis of Lemma~\ref{lemAlgo}.
Line~\ref{algoBinSearchLmdUpdate} updates $\lambda$ as a mean of the upper-bound $\overline{\lambda}$ and the lower-bound $\underline{\lambda}$,
and then Lines~\ref{algoBinSearchInvStart}--\ref{algoBinSearchInvEnd} approximately calculate $\mu^{'-1}(\lambda)$ on the basis of the samples $\varepsilon_i$ for $i=1,2,\dots,Q$
and the following lemma:
\begin{lem}\label{lemInS}
	For $\varepsilon \in \bR^d$, $\varepsilon \in S(\lambda, \cY; \Sigma)$ if and only if $\psi(\mu, \varepsilon) \geq 0$, where
	\begin{align}
		\psi(\mu, \varepsilon) := \min\{ \psi_{1,+}(\varepsilon),\psi_{1,-}(\varepsilon),\dots,\psi_{K,+}(\varepsilon),\psi_{K,-}(\varepsilon) \} \nonumber  \\
		\text{ where } \psi_{k,+}(\varepsilon) := \min_{y \in \cY} g_k(y, \varepsilon)  + \mu r_k(y; \Sigma), \nonumber\\
		\psi_{k,-}(\varepsilon) := \min_{y \in \cY} g_k(y, -\varepsilon)  + \mu r_k(y; \Sigma).  \label{optViolation}  
	\end{align}
\end{lem}
For each $\varepsilon_i$,
Line~\ref{algoOpt} calculates the optimization problem \eqref{optViolation} for $\varepsilon = \varepsilon_i$,
and then Lines~\ref{algoBinIfPsiStart}--\ref{algoBinIfPsiEnd} approximately calculate $\mu^{'-1}(\lambda)$
by counting the number of such $\varepsilon_i \in S(\lambda, \cY; \Sigma)$ for $i=1,2,\dots,Q$.
Note that some of the calculation of $\psi(\lambda,\varepsilon_i)$ can be omitted in practice 
since $\psi(\lambda,\varepsilon_i)$ is monotonically increasing in $\lambda$.
Lines~\ref{algoBinUpdateStart}--\ref{algoBinUpdateEnd} update the upper-bound $\overline{\lambda}$ or the lower-bound $\underline{\lambda}$
on the basis of the approximate probability $q$ of $\mu^{'-1}(\lambda)$.
Finally, Line~\ref{algoSamplingEnd} outputs the approximate upper-bound $\dot{\mu} = \overline{\lambda}$ of $\mu'(p)$.
\begin{algorithm}[t]
	\caption{Estimation of $\mu(p, \cY; \cN_{\Sigma},  \Sigma)$}\label{algoSample}
	\begin{algorithmic}[1]
		\REQUIRE $p$, $\cY$, $\Sigma$, and accuracy parameter $(\alpha,\beta,\gamma)$.
		\ENSURE An estimate $\dot{\mu}$ of $\mu' (p, \cY; \cN_{\Sigma}, \Sigma)$
		\STATE Define the number of sample $Q$ by \eqref{defNumSample}. \label{algoDefQ}
		\STATE Sample $\varepsilon_i \sim \cN(0, \Sigma)$ for $i=1,2,\dots,Q$. \label{algoSampling}
		\STATE \# Binary search \label{algoBinSearchStart}
		\STATE Initialize $\underline{\lambda} = \chi_1^{-1}(p)$, $\overline{\lambda}= \chi_d^{-1}(p)$ \label{algoBinSearchInit}
		\WHILE{$\overline{\lambda} - \underline{\lambda} \geq \gamma $}
		\STATE Put $\lambda = (\underline{\lambda} + \overline{\lambda}) / 2$ and initialize $q = 0$ \label{algoBinSearchLmdUpdate}
		\FOR{i=1,2,\dots,Q} \label{algoBinSearchInvStart}
		\STATE Solve \eqref{optViolation} and obtain $\psi(\mu, \varepsilon_i)$.\label{algoOpt}
		\IF{$\psi(\mu, \varepsilon_i) \geq 0$} \label{algoBinIfPsiStart}
		\STATE $q = q + 1/Q$
		\ENDIF\label{algoBinIfPsiEnd}
		\ENDFOR \label{algoBinSearchInvEnd}
		\IF{$q \geq p + (\beta / 2)$}\label{algoBinUpdateStart}
		\STATE Put $\underline{\lambda} = \lambda$.
		\ELSE 
		\STATE Put $\overline{\lambda} = \lambda$.
		\ENDIF\label{algoBinUpdateEnd}
		\ENDWHILE \label{algoBinSearchEnd}
		\STATE Output $\dot{\mu} =  \overline{\lambda}$ \label{algoSamplingEnd}
	\end{algorithmic}
\end{algorithm}

This output $\dot{\mu}$ satisfies the following probabilistic guarantee.
\begin{prop}\label{propSampleAlgo}
	With probability at least $1- \alpha$,
	the output $\dot{\mu} $ of Algorithm~\ref{algoSample} satisfy
	\begin{align*}
		\mu'(p) \leq  \dot{\mu} \leq \mu'(p + \beta) + \gamma.
	\end{align*}
\end{prop}
The computational tractability of Algorithm~\ref{algoSample} relies heavily on
the choice of $\cY$,
and this is discussed in the next section, which utilizes Algorithm~\ref{algoSample}.

\section{Empirical Domain Reduction}\label{secReduction}
%
This section presents an algorithm that achieves our main theoretical result
via the concept of subspace selection discussed in Section~\ref{subsecApproach}.
\subsection{Two-step domain reduction algorithm}
%
Let us define constants $\alpha_n, \beta_n, \gamma_n, \eta_n, \delta_n >0$ and empirically reduced domain $\hat{\cY}(w, \mu)$ for $w,\mu \geq 0$ by
\begin{align*}
	\alpha_n = \frac{\delta}{\sqrt{n}}, \quad  \beta_n = \frac{c_d \tau }{\sqrt{n}}, \quad \gamma_n = \frac{1}{\sqrt{n}},\\
	\eta_n := \eta_{n,d}\left(\frac{d^2 c_1 \tau'}{\sqrt{n}} \right),\\
	\delta_n :=  \frac{1}{\sqrt{n}}\left(2 d^2 c_1 \tau'  + 4c_d \tau  + \delta \right),\\
	\hat{\cY}(w, \mu) := \left\{y \in \cX \left| 
	\begin{array}{l}
		f(y) \leq w,\\
		g_k (y, \hat{\theta}) \geq - \mu r_k (y; \hat{\Sigma}) , \\
		k=1,2,\dots,K 
	\end{array}
	\right\} \right. .
\end{align*}
We can then propose Algorithm~\ref{algoReduction} for the calculation of the scale $\hat{\lambda}$ of robustness.
Lines~\ref{algoRedstep1}--\ref{algoRedDom} represent the first stage of this algorithm,
which conducts subspace selection.
Line~\ref{algoRedstep1} first calculates coefficient $\dot{\mu}$ using Algorithm~\ref{algoSample} with $\cY=\cX$.
Line~\ref{algoRedMuScale} defines $\hat{\mu}$ by scaling $\dot{\mu}$,
and then Line~\ref{algoRedsolve} calculates optimum value $\hat{w} = f(\hat{x}(3 \hat{\mu}))$ by solving \eqref{robustOpt} on the basis of scale $\hat{\mu}$.
Line~\ref{algoRedDom} conduct subspace selection by defining $\hat{\cY} = \hat{\cY}(\hat{w}, \hat{\mu} / (1 - \eta_n))$.
Lines~\ref{algoRedStep2}--\ref{algoRedOut} represent the second stage of the algorithm,
which calculates scale $\hat{\lambda}$ on the basis of the subspace $\hat{\cY}$.
Line~\ref{algoRedStep2} calculates $\dot{\lambda}$ using Algorithm~\ref{algoSample} with $\cY=\hat{\cY}$,
and Line~\ref{algoRedOut} then outputs $\hat{\lambda}$ which defined by scaling $\dot{\lambda}$.
\begin{algorithm}[t]
	\caption{Computation of empirical scale $\hat{\lambda}$ for minimizing $\VaR_{\alpha}$}\label{algoReduction}
	\begin{algorithmic}[1]
		\REQUIRE $\delta$, $\hat{\theta}_n$, $\hat{\Sigma}_n$
		\ENSURE $x \in \cX$
		\STATE Calculate $\dot{\mu}$ with Algorithm~\ref{algoSample} with inputs $p= 1 - 2c_d \tau /\sqrt{n}$, $\cY = \cX$, $\Sigma =  \hat{\Sigma}_n$, and $(\alpha, \beta, \gamma) = (\alpha_n, \beta_n, \gamma_n)$ \label{algoRedstep1}
		\STATE Define $\hat{\mu} = (1+\eta_n) \dot{\mu} / (1 - \eta_n)^2 \sqrt{n}$ \label{algoRedMuScale}
		\STATE Calculate $\hat{w} = f( \hat{x}( 3\hat{\mu} )  )$ \label{algoRedsolve}
		\STATE Define $\hat{\cY} := \hat{\cY}(\hat{w}, \hat{\mu} / (1 - \eta_n))$ \label{algoRedDom}
		\STATE Calculate $\dot{\lambda}$ with Algorithm~\ref{algoSample} with inputs $p= 1 - \delta + \delta_n$, $\cY = \hat{\cY}$, $\Sigma = \hat{\Sigma}_n$, and  $(\alpha, \beta, \gamma) = (\alpha_n, \beta_n, \gamma_n)$ \label{algoRedStep2}
		\STATE Output $\hat{\lambda}:= \dot{\lambda}/ (1-\gamma_n) \sqrt{n}$. \label{algoRedOut}
	\end{algorithmic}
\end{algorithm}
\subsection{Theoretical analysis}
Let $n_0$ be the minimum integer $n$ that satisfies
\begin{align*}
	\frac{(1 + \eta_n)^2}{( 1 - \eta_n)^2} \mu(1 - \delta + \delta_n, \cX, \cN_{\Sigma^*}, \Sigma^*)  + \frac{1 + \eta_n}{ (1 - \eta_n) \sqrt{n} } 
	\leq \mu(1 - 2c_d \tau /\sqrt{n}, \cX, \cN_{\Sigma^*}, \Sigma^*)
\end{align*}
Note that such $n_0$ must exist since $\delta_n, \eta_n = O(1/\sqrt{n})$.
Our main theoretical result can then be given as follows:
\begin{thm}\label{thmFormal}
	Suppose that $n \geq n_0$,
	and $\lim_{\varepsilon \to +0} \cY^*(f(x^*(0)) + \varepsilon, \varepsilon) = \{x^*(0)\}$.
	Then the output $\hat{\lambda}$ of Algorithm~\ref{algoReduction}
	satisfies
	\begin{align}
		\VaR_{\delta}[f^*(\hat{x}_n(\hat{\lambda}))] \leq f(x^*(\overline{\lambda}_n)) \label{ineqThmMain1}
	\end{align}
	where
	\begin{align}
		\lim_{n \to \infty} \sqrt{n}  \overline{\lambda}_n \leq   \chi_{1}^{-1}(1 - \delta/K). \label{ineqThmMain2}
	\end{align}
\end{thm}
Thus, under the uniqueness of the true optimum solution $x^*(0)$,
the estimated scale $\hat{\lambda}$ is asymptotically independent of the dimension $d$.

The essence of the proof of this statement lies in 
finding subspaces $\underline{\cY}_n, \overline{\cY}_n \subseteq \cX$
that satisfy the following two conditions:
(i) with high probability, it holds that 
\begin{align*}
	\hat{x}(\lambda_n), x^*(2\lambda_n) \in \underline{\cY}_n \subseteq \hat{\cY} \subseteq \overline{\cY}_n,
\end{align*}
and (ii) $\lim_{n \to \infty} \overline{\cY}_n = \{x^*(0) \}$.
Property (i), together with Lemma~\ref{lemConf},
implies contribution to proof of \eqref{ineqThmMain1},
and the property (ii) implies asymptotic convergence \eqref{ineqThmMain2}.
Since $\hat{\cY}$ is controlled by the estimated optimum value $\hat{w}$,
for the existence of both $\underline{\cY}_n$ and $\overline{\cY}_n$,
we need to control the range of estimated optimum value $\hat{w}$.
We utilize the following lemma for this control:
\begin{lem}\label{lemConf2}
	Let $\mu \geq 0$ and $\kappa \geq 1$. Suppose that $\theta \in \bR^d$, $\Sigma \succeq O$, and $\cY \subseteq \bR^d$ satisfy
	\begin{align*}
		\theta - \theta^* \in S(\mu, \cY; \Sigma),  \\
		x( (\kappa-1) \mu ; \theta^*, \Sigma ), x(\kappa \mu; \theta, \Sigma), x((\kappa + 1) \mu; \theta^*, \Sigma) \in \cY. 
	\end{align*}
	It then holds that $f^*(x((\kappa-1) \mu; \theta^*, \Sigma) ) \leq  f^*(x(\kappa \mu; \theta, \Sigma) ) \leq f^*(x((\kappa + 1)\mu; \theta^*, \Sigma))$.
\end{lem}
Roughly speaking, we bound the range of $\hat{w}$ in probability
by applying Lemma~\ref{lemConf2} with $\kappa = 3$ and $\mu = \hat{\mu}$.

Let us conclude this section 
with a brief discussion of the computational tractability of our algorithms.
\begin{rmk}
	Algorithm~\ref{algoReduction} mainly consists of 
	an optimization~\eqref{algoRedsolve} in Line~\ref{algoRedsolve} 
	and two applications of Algorithm~\ref{algoReduction} in Lines~\ref{algoRedstep1} and~\ref{algoRedStep2}, where each application consists of
	$O(n \log^2 n)$ optimization of \eqref{optViolation}.
	If the original robust optimization~\eqref{algoRedsolve} is a convex programming problem,
	then \eqref{optViolation} with $\cY = \cX$ is a series of $K$ convex programmings.
	It is thus natural to assume that Lines~\ref{algoRedstep1} and~\ref{algoRedsolve} are computationally tractable.
	
	For the second application of Algorithm~\ref{algoSample}, in Line~\ref{algoRedStep2},
	note that the subspace $\cY = \hat{\cY}$ will not generally be convex
	because of the concavity of $-r_k$.
	In practical implementation,
	we can replace $r_k$ in the definition of $\hat{\cY}(w, \mu)$
	by an upper-bound $u_k$ that makes $\hat{\cY}$ convex.
	If such a $u_k$ can be bounded above as $u_k(x ; \Sigma) \leq \kappa r_k(x ; \Sigma) $ with some constant $\kappa \geq 1$, 
	then the resulting output $\hat{\lambda}$ of Algorithm~\ref{algoReduction}
	will also satisfy the guarantee in Theorem~\ref{thmFormal}.
	
	Let us introduce a simple example of such an upper-bound $u_k$.
	Suppose that $\cX$ is the positive orthant ($\cX = \bR^d_{+}$) and that 
	$g_k(x, \theta) = \theta^{\top} x$.
	Then $r_k(x, I_d) = \|x \|_2$, where $\| \cdot \|_q$ is the $\ell_q$-norm for $q \geq 0$.
	It then holds that $r_k(x, I_d) \leq \sqrt{d} \|x \|_1$.
	Observe that $\| \cdot \|_1$ is linear on the positive orthant $\cX$,
	and thus this $u_k$ is computationally tractable.
\end{rmk}

\section{Experiments}

\begin{figure*}[t]
	\begin{minipage}[t]{0.485\hsize}
		\begin{center}
			\includegraphics[width=\hsize]{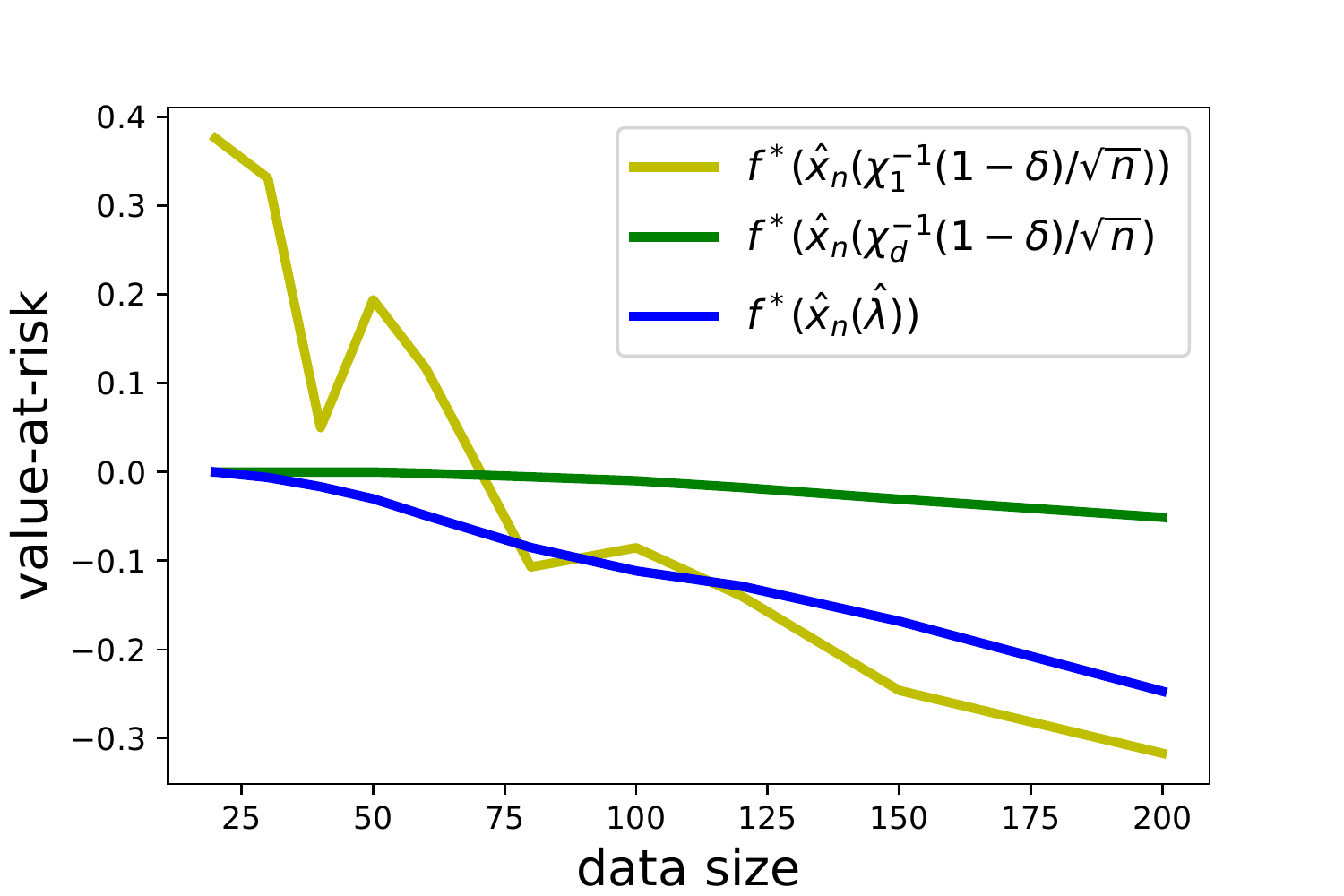}
			\caption{Value-at-risk $\VaR[f^*(\hat{x}_n(\lambda))]$ for scale $\lambda$. The vertical line shows the value-at-risk, and the horizontal line shows the number of samples. Respective yellow, green, blue lines show the result with $\lambda = \chi_1^{-1}(1-\delta)$, $\chi_d^{-1}(1-\delta)$, and $\hat{\lambda}_n$.} \label{figGenError}
		\end{center}
	\end{minipage}
	\hspace{0.02\hsize}
	\begin{minipage}[t]{0.485\hsize}
		\begin{center}
			\includegraphics[width=\hsize]{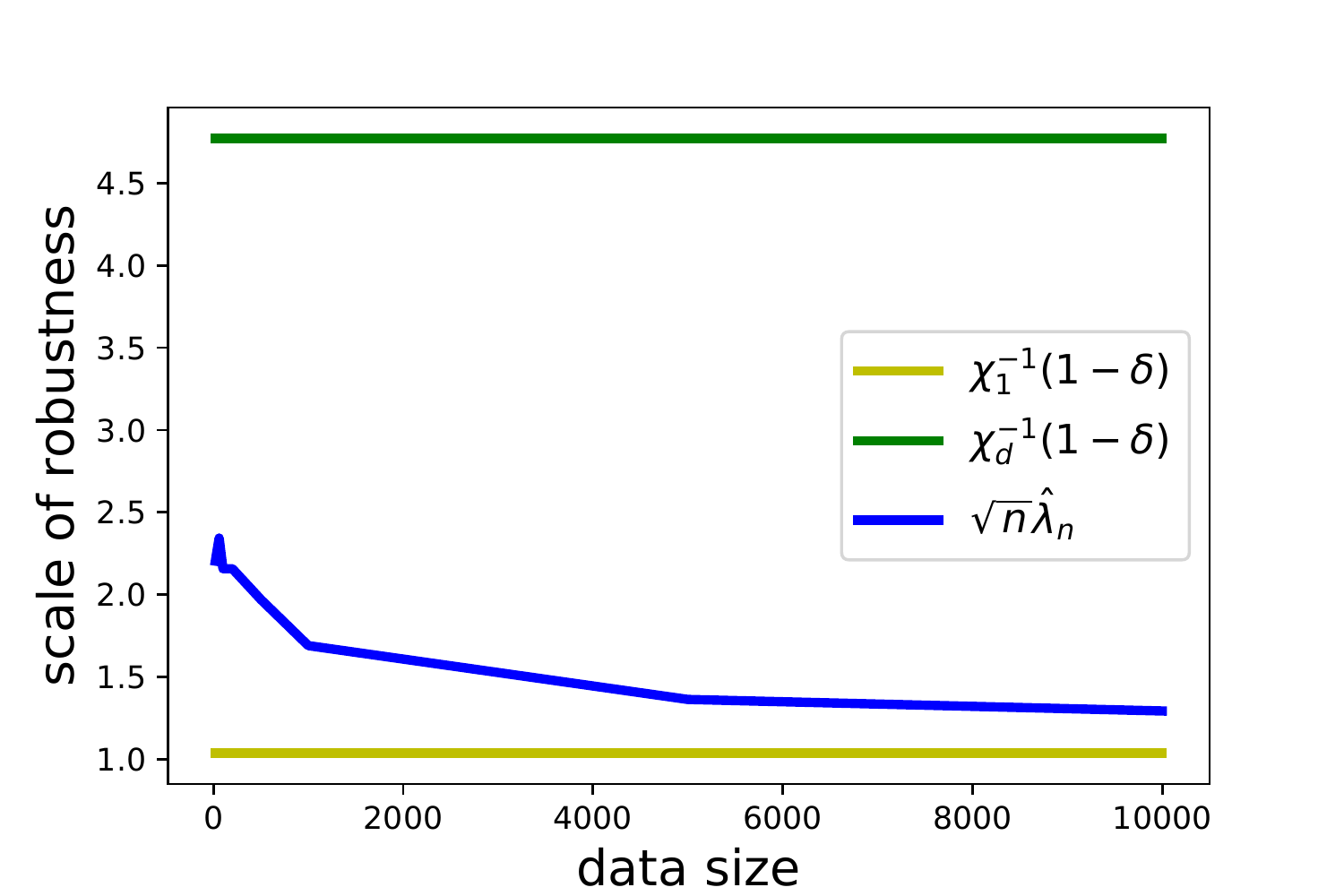}
			\caption{Scale of robustness $\lambda$. Vertical line shows the value of scale of robustness $\lambda$, and the horizontal line shows the number of samples. Respective yellow, green, blue lines indicate the value of $\lambda = \chi_1^{-1}(1-\delta)$, $\chi_d^{-1}(1-\delta)$, and $\sqrt{n}\hat{\lambda}_n$.}\label{figRegScale}
		\end{center}
	\end{minipage}
\end{figure*}

\subsection{Experimental setting}
Let us consider the following simple portfolio optimization problem, which has been examined by several existing studies in robust optimization~\cite{bertsimas2004price,ben1999robust}.
Suppose that there are $d$ items whose costs are $\theta_1, \dots, \theta_d$, which are uncertain parameters.
The task is to find a convex combination (portfolio) that minimizes the cost:
\begin{align}
	\begin{array}{ll}
		& \displaystyle \min \sum_{i=1}^d \theta_i x_i \\
		& \text{s.t. } \displaystyle \sum_{i=1}^d x_i \leq 1, \quad 
		x_i \geq 0, \ \ i=1,\dots,d.
	\end{array} \label{portfolioOpt} 
\end{align}
This problem has an uncertain objective function.
Thus, we first translate this problem into the our setting~\eqref{nominalOpt}.
Let us introduce an additional variable $x_0$.
We denote by $x$ the set of variables $x=(x_0,x_1,\dots,x_d)$ over $\bR^{d+1}$.
We also introduce a dummy parameter $\theta_0 = 1$.
We define $\cX$,
$f$, and $g$ by
\begin{align*}
	\cX &:= \left\{x \in \bR^{d+1}
	\left|
	\begin{array}{l}
		\sum_{i=1}^d x_i \leq 1  \\
		x_i \geq 0, \quad i=1,2,\dots,d.
	\end{array}
	\right\}, \right. \\
	f(x) &:= x_0,\\
	g(x, \theta) &:= \theta_0 x_0 - \sum_{i=1}^d \theta_i x_i.
\end{align*}
Then the portfolio optimization problem~\eqref{portfolioOpt} is captured by our setting \eqref{nominalOpt}.

For this problem, we generate synthetic data as follows.
We fix $\delta=0.3$, $d = 20$, and $\theta^* = (-1, -0.9, -0.8,\dots,0.8,0.9)^{\top}$.
We set the covariance matrix $\Sigma^*$ as a diagonal matrix $\Sigma^* = \mathrm{diag}(\sigma_1^2, \sigma_2^2,\dots,\sigma_d^2)$
where $\sigma_i$ follows a uniform distribution over $[0,10]$.
For each sample size $n$ and true covariance matrix,
we generate $20$ sets of $n$ i.i.d. samples from  $\cN(\theta^*, \Sigma^*)$,
over which samples value-at-risk will be approximately calculated.
All our experimental results are average over $30$ generation of such covariance matrix $\Sigma^*$.
For a solution $\hat{x}$ that does not satisfy the true constraint,
we put $f^*(\hat{x}) = 1$ instead of $f^*(\hat{x}) = \infty$ for the purpose of clear visualization.

To make Algorithm~\ref{algoReduction} practical,
we applied the following approximations:
While we bound the gap of the probability mass of the true distribution of $\hat{\theta}_n - \theta^*$ and $\cN_{\Sigma^*/n}$
by $c_d \tau / \sqrt{n}$ in our theoretical analysis,
we simply replace the true distribution by its normal estimate $\cN_{\hat{\Sigma}_n/n}$ in implementation by putting $c_d = \tau = 0$.
Also, our theoretical analysis carefully replace the true covariance matrix $\Sigma^*$
by its probabilistic upper-bound $\hat{\Sigma}_n / (1 - \eta_n)$ or lower-bound $\hat{\Sigma}_n / (1 + \eta_n)$,
but we simply replace $\Sigma^*$ by $\hat{\Sigma}_n$ in implementation by putting $\tau' = 0$.
\subsection{Experimental results}
We solved the robust optimization problem~\eqref{robustOpt} for three different $\lambda$:
\begin{itemize}
	\item {}[proposed] $\lambda = \hat{\lambda}_n$ calculated by Algorithm~\ref{algoReduction},
	\item {}[lower-bound] $\lambda = \underline{\lambda} := \chi_1^{-1}(1 - \delta) / \sqrt{n}$. This is a lower-bound in terms of Lemma~\ref{lemAlgo}.
	\item {}[upper bound] $\lambda = \overline{\lambda} := \chi_d^{-1}(1 - \delta)/ \sqrt{n}$.
	This is the the scale adopted by the standard approach on the basis of Fact~\ref{factConf}.
\end{itemize}
Figure~\ref{figGenError} shows the value-at-risk $\VaR_{\delta}[f^*(\hat{x}(\lambda))]$ for each $\lambda$.
We observed that
\begin{itemize}[noitemsep,nolistsep,leftmargin=*]
	\item The lower-bound $\underline{\lambda}$ (yellow) showed worst performance with small sample size $n \leq 60$.
	This performance is due to the violation of true constraints with probability more than $\delta$.
	Although this $\underline{\lambda}$ showed the best performance with large sample size $n \geq 120$,
	this scale is inadequate for guaranteed optimization.
	\item The upper-bound $\overline{\lambda}$ (green) could always satisfy the true constraint
	with designed probability $1-\delta$, but showed the worst speed of convergence.
	This shows that $\overline{\lambda}$ is too conservative.
	\item The proposed method $\hat{\lambda}_n$ (blue) showed the best performance with wide range of sample size $20 \leq n \leq 120$.
	This can always satisfy the true constraint,
	and, at the same time, showed rather speedy convergence than $\overline{\lambda}$.
\end{itemize}

We next observe the convergence of the scale $\sqrt{n}\hat{\lambda}_n$ to $\chi_1^{-1}(1-\delta)$,
which is proven by Theorem~\ref{thmFormal}.
Figure~\ref{figRegScale} show the value $\sqrt{n}\hat{\lambda}_n$ (blue) 
in comparison with its lower-bound $\chi_1^{-1}(1-\delta)$ (yellow) and upper-bound $ \chi_d^{-1}(1-\delta)$ (green).
The scale $\sqrt{n}\hat{\lambda}_n$ is about $\chi_4^{-1}(1-\delta) \approx 2.2$ with small samples $n \leq 200$,
and this implies that 
our algorithm succeeded in selecting four dimensional subspace.
The scale $\sqrt{n}\hat{\lambda}_n$ decreased to $\chi_2^{-1}(1-\delta) \approx 1.55$ with $n = 1000$,
and then slowly converged to $\chi_1^{-1}(1- \delta)$.

\section{Concluding remarks}
This paper has shown that the scale of robustness required for achieving a given confidence probability
is dependent on the size of optimization domain,
and then has proposed an algorithm for deciding the scale tightly
on the basis of empirical domain reduction.
The analysis has proven that the scale is asymptotically independent of the parameter dimension,
and our experiments has demonstrated that
the proposed method can achieve better performance while maintaining probabilistic guarantee.

\bibliography{reference}
\bibliographystyle{plain}

\normalsize
\appendix

\section{Proofs}
\subsection{Proofs for the statements in Section~\ref{secIntro}}
\begin{proof}[Proof of Fact~\ref{factConf}]
	By $\theta - \theta^* \in \lambda U_{\Sigma}$, for any $x \in \cX$ and $k=1,2,\dots,K$, it holds that
	\begin{align*}
		\min_{u \in U_{\Sigma}} g_k (x, \theta^* + 2\lambda u ) \leq \min_{u \in U_{\Sigma}} g_k (x, \theta + \lambda u ) \leq g_k (x, \theta^*). 
	\end{align*}
	The left inequality implies that $f(x(\lambda; \theta,\Sigma)) \leq f(x(2\lambda; \theta^*,\Sigma))$,
	and the right inequality implies that $f^*(x(\lambda; \theta,\Sigma)) < \infty$.
	It thus holds that $f^*(x(\lambda; \theta,\Sigma)) \leq f^*(x(2\lambda; \theta^*,\Sigma))$.
\end{proof}
\begin{proof}[Proof of Proposition~\ref{propBase}]
	Since $\hat{\theta}_n - \theta^* \in \lambda U_{\Sigma^*}$ with probability at least $1-\delta$, it holds that
	\begin{align*}
		\VaR_{\delta}[f^*(x(\lambda ; \hat{\theta}_n, \Sigma^*) )] &\leq \max_{\theta : \theta- \theta^* \in \lambda U_{\Sigma^*}} f^*(x(\lambda ; \theta, \Sigma^*) )   \\
		&\leq f^*(x^*(2 \lambda )).
	\end{align*}
	The last inequality follows from Fact~\ref{factConf}.
\end{proof}
\begin{proof}[Proof of Lemma~\ref{lemConf}]
	By the definition of $r_k$ and $S(\lambda, \cY; \Sigma)$ and linearlity of $g_k$ with respect to $\theta$, for all $y \in \cY$, it holds that
	\begin{align*}
		\min_{u \in U_{\Sigma}} g_{k}(y, \theta^* + 2\lambda u)
		&= g_{k}(y, \theta) - g_{k}(y, \theta^* - \theta) - 2 \lambda r_k(y;\Sigma)\\
		&\leq \min_{u \in U_{\Sigma}} g_k (y, \theta + \lambda u ) \\
		&= g_k(y, \theta^*) + g_k(y, \theta - \theta^*)  - \lambda r_k (y; \Sigma)\\
		&\leq g_k(y, \theta^* ).
	\end{align*}
	Since $x(\lambda; \theta, \Sigma), x(2\lambda; \theta^*, \Sigma) \in \cY$,
	it holds that $f^*(x(\lambda; \theta, \Sigma)), f^*(x(2\lambda; \theta^*, \Sigma) ) < \infty$.
	By the first inequality and the optimality of $x(\lambda; \theta, \Sigma)$ over $\cX$, then,
	we have $f(x(\lambda; \theta, \Sigma)) \leq f(y)$ that satisfies $y \in \cY$ and $\min_{u \in U_{\Sigma}} g_{k}(y, \theta^* + 2\lambda u) \geq 0$.
	Thus the statement holds.
\end{proof}

\subsection{Proofs for the statements in Section~\ref{secPre}}
For each $\theta^i \sim \cP^*$ for $i=1,2,\dots,n$,
let us denote $\Delta^i = \Sigma^{*-1/2} (\theta^i - \theta^*)$ and $\Delta_n = \sum_{i=1}^n \Delta^i / \sqrt{n}$.
Then each $\Delta^i$ is distributed by $\cP_{\Delta}^{*}$.
\begin{proof}[Proof of Lemma~\ref{lemParaEst}]
	Observe that each $\Delta^i$ satisfies $E[\Delta^i ]=0$,
	$E[\Delta^i \Delta^{i \top}] = I_d$, and $E[\|\Delta^i \|^3]  \leq \tau$ by Assumption~\ref{assump}.
	Lemma~\ref{lemHighBE} implies that
	\begin{align*}
		\sup_{C \in \cC_d} \left| \underset{W \sim \cN(0,I_d)}{\Prob} (W \in C) - \underset{\Delta_n}{\Prob} (\Delta_n \in C) \right| \leq \frac{c_d \tau }{\sqrt{n}}.
	\end{align*}
	Mapping both random variables as $W \mapsto (\Sigma^* / n)^{1/2} W$ and $\Delta_n \mapsto (\Sigma^* / n)^{1/2}\Delta_n = \hat{\theta}_n$,
	we have the statement.
\end{proof}
\begin{proof}[Proof of Lemma~\ref{lemCovEst}]
	Let us define $\hat{I}_n := (1/n) \sum_{i=1}^n \Delta^i \Delta^{i \top}$.
	Since $\Sigma^{1/2} \hat{I}_n \Sigma^{1/2} = \hat{\Sigma}_n$ and 
	the relationship $O \preceq \Sigma_1 \preceq \Sigma_2 $ of positive semidefinit matrices 
	is invariant under the mapping $\Sigma_1 \mapsto \Sigma^{1/2} \Sigma_1 \Sigma^{1/2}$ by another positive semidefinite matrix $\Sigma$,
	it is enough to prove that
	\begin{align}
		\left(1-  \frac{c_d \tau}{\sqrt{n}} \right) I_d
		\preceq \hat{I}_n
		\preceq \left( 1+ \frac{c_d \tau}{\sqrt{n}} \right) I_d.
	\end{align}
	
	For each $i=1,2,\dots,n$ and $j, k=1,2,\dots,d$, let us define $\sigma^{i}_{jk} = \Delta^i_{j} \Delta^i_{k} $.
	Then $E[\sigma^{i}_{jj}] = 1$, $\sigma^{*2}_{jj} =  E[\sigma^{i2}_{jj}] \leq \tau^{'2/3}$ by H\"older's inequality,
	and $E[\sigma^{i3}_{jj}] \leq \tau'$.
	Observe that $\hat{I}_{n,jj} = (1/n)(\sum_{i=1}^n \sigma^{i}_{jj} )  - 1$
	Thus, by Lemma~\ref{lemHighBE}, it holds that
	\begin{align*}
		\sup_{\lambda \geq 0} \left| 
		\underset{W \sim \cN(0,\sigma^{*2}_{jj} / n )}{\Prob} (W \in [-\lambda, \lambda])   
		- \underset{\hat{I}_n}{\Prob} ( \hat{I}_{n,jj} \in [-\lambda +1, \lambda +1]  )
		\right| 
		\leq \frac{c_1 \tau'}{\sqrt{n}}.
	\end{align*}
	Since $\Prob_{W \sim \cN(0,\sigma^{*2}_{jj} )} (W \in [-\lambda, \lambda] ) = \chi_{1} (\lambda / \sigma^{*}_{jj} )$,
	the following holds with probability $1 - \delta'/d^2$, it holds that
	\begin{align*}
		|\hat{I}_{n,jj} - 1| \leq \frac{\tau^{'1/3}}{\sqrt{n}} \chi_1^{-1}\left( 1 - \frac{\delta'}{d^2}  + \frac{c_1 \tau'}{ \sqrt{n}}  \right) = \frac{\eta_{d,n}}{d}.
	\end{align*}
	Also, for $j\neq k$, it holds that $E[\sigma^{i}_{jk}] = 0$, $\sigma^{*2}_{jk} := E[\sigma^{i2}_{jk}]  \leq \tau^{'2/3}$ by H\"older's inequality,
	and $E[\sigma^{i3}_{jk}] \leq \tau'$ also by H\"older's inequality.
	Thus, by Lemma~\ref{lemHighBE}, it holds that
	\begin{align*}
		\sup_{\lambda \geq 0} \left| \underset{W \sim \cN(0,\sigma^{*2}_{jk} / n )}{\Prob} (W \in [-\lambda, \lambda]) 
		- \underset{\hat{I}_n}{\Prob} \left( I_{n,jk} \in [-\lambda , \lambda]  \right) \right| 
		\leq \frac{c_1 \tau'}{\sqrt{n}}.
	\end{align*}
	The following holds with probability $1 - \delta'/d^2$, it holds that
	\begin{align*}
		|\hat{I}_{n,jk}| \leq \eta_{d,n}/ d.
	\end{align*}
	Let us define 
	By the union bound, then, $|\hat{I}_n{jk} -  I_{d, jk}| \leq \eta_{d,n}/ d$ for all $j,k=1,2,\dots,d$ with probability at least $1 - \delta'$.
	
	Suppose that  $|\hat{I}_{n,jk} -  I_{d, jk}| \leq \eta$ for all $j,k=1,2,\dots,d$.
	By Lemma~\ref{lemGersh}, then, it holds that
	\begin{align*}
		\left( 1 -  \eta_{d,n} \right) I_{d } \preceq \hat{I}_n \preceq \left( 1 +  \eta_{d,n} \right) I_{d }.
	\end{align*}
	Thus the statement holds.
\end{proof}

\subsection{Proofs for the statements in Section~\ref{secSample}}
\begin{proof}[Proof of Lemma~\ref{lemSConvex}]
	Let us define $S^+_k$ and $S^-_k$ for $k=1,2,\dots,K$ by
	\begin{align}
		S^+_k(\lambda, \cY; \Sigma) := \{\Delta \in \bR^d  \mid \forall y \in \cY,  g_k (y, - \Delta) \leq \lambda r_k(y; \Sigma) \}, \label{dfS+k} \\
		S^-_k(\lambda, \cY; \Sigma) := \{\Delta \in \bR^d  \mid \forall y \in \cY, g_k (y, \Delta) \leq \lambda r_k(y; \Sigma) \}. \label{dfS-k}
	\end{align}
	Let us denote $S^+_k = S^+_k(\lambda, \cY; \Sigma)$ and $S^-_k = S^-_k(\lambda, \cY; \Sigma)$.
	It then holds that $S(\lambda, \cY; \Sigma) = \bigcap_{k=1}^K (S^+_k \cap S^-_k)$.
	We then show that $S^+_k$ and $S^-_k$ are convex sets for all $k=1,2,\dots, K$.
	
	Observe that,
	\begin{align*}
		\Delta \in S^+_k &\Leftrightarrow \forall y \in \cY, g_k (y, - \Delta) \leq \lambda r_k(y; \Sigma)\\
		&\Leftrightarrow \forall y \min_{u \in U_{\Sigma}} g_k (y, -\Delta + \lambda u) \leq 0,\\
		&\Leftrightarrow \forall y ,\exists u' \in \lambda U_{\Sigma}, g_k (y, -\Delta + u') \leq 0
	\end{align*}
	For each $\Delta \in S^+_k$ and $y \in \cY$, then, let $u^+_{k}(\Delta, y)$ denote an element of $\lambda U_{\Sigma}$ that satisfy $g_k (y, -\Delta + u^+_{k}(\Delta, y)) \leq 0$.
	Suppose that $\Delta_1, \Delta_2 \in S^+_k$, and we will prove that $\alpha \Delta_1 + (1 - \alpha) \Delta_2 \in S^+_k$ for any $\alpha \in [0,1]$.
	For any $y \in \cY$,  let $u' := \alpha u^+_{k}(\Delta_1, y) + (1 - \alpha) u^+_{k}(\Delta_2, y)$, which $u'$ is included by $\lambda U_{\Sigma}$ by the convexity of $U_{\Sigma}$.
	Then, by the linearlity of $g_k$ with respect to $\theta$, we have
	$g_k (y, -(\alpha \Delta_1 + (1 - \alpha) \Delta_2) + u') = \alpha g_k (y, - \Delta_1 + u^+_{k}(\Delta_1, y)) + (1-\alpha) g_k (y, - \Delta_2 + u^+_{k}(\Delta_2, y)) \leq 0$, and thus $\alpha \Delta_1 + (1 - \alpha) \Delta_2 \in S^+_k$ and $S^+_k$ is convex set.
	The same line of discussion shows that $S^-_k$ is also a convex set, and thus the first half of the statement holds.
	
	Assume in addition that $g_k(x,\theta)$ is linear with respect to $x$ for all $k=1,2,\dots,K$ and $\cY$ is a convex set.
	By Sion's minimax theorem, then, it holds that
	\begin{align*}
		\Delta \in S^+_k &\Leftrightarrow \forall y \in \cY, g_k (y, - \Delta) \leq \lambda r_k(y; \Sigma)\\
		&\Leftrightarrow \max_{y \in \cY} \min_{u' \in \lambda U_{\Sigma}}  g_k (y, -\Delta + u') \leq 0\\
		&\Leftrightarrow \min_{u' \in \lambda U_{\Sigma}} \max_{y \in \cY} g_k (y, -\Delta + u') \leq 0\\
		&\Leftrightarrow \exists u' \in \lambda U_{\Sigma},  \Delta + u' \in C^+_k \\
		&\Leftrightarrow  \Delta \in \left(\lambda U_{\Sigma} +  C^+_k \right).
	\end{align*}
	It then also holds that $\Delta \in S^-_k$ if and only if $\Delta \in \left(\lambda U_{\Sigma} +  C^-_k \right)$.
	The second half of the statement thus holds, and the proof is complete.
\end{proof}

\begin{proof}[Proof of Lemma~\ref{lemMonotonicity}]
	(i) This holds since $S(\lambda_1, \cY; \Sigma ) \subseteq S(\lambda_2, \cY; \Sigma) )$ for any $0 \leq \lambda_1 \leq \lambda_2$.
	
	(ii) This holds since $S(\lambda, \cY; \Sigma ) \supseteq S(\lambda, \cZ; \Sigma) )$ for any $\cY \subseteq \cZ \subseteq \cX$.
	
	(iii) This holds since $S(\lambda, \cY; \Sigma_1 ) \subseteq S(\lambda, \cY; \Sigma_2 )$ for $O \preceq  \Sigma_1 \preceq \Sigma_2$. 
	
	(iv) Recall the definitions \eqref{dfS+k} and \eqref{dfS-k} of $S^+_k(\lambda, \cY; \Sigma)$ and $S^-_k(\lambda, \cY; \Sigma)$.
	We first prove that $S^+_k(\kappa_1 \lambda, \cY; \kappa_2 \Sigma) = \kappa_1 \sqrt{\kappa_2} S^+_k(\lambda, \cY; \Sigma)$ for $\kappa_1, \kappa_2 > 0$.
	By the definition, we have
	\begin{align*}
		&S^+_k( \kappa_1 \lambda, \cY; \kappa_2 \Sigma)\\ 
		&= \{\Delta \in \bR^d  \mid \forall y \in \cY,  g_k (y, - \Delta) \leq \kappa_1 \lambda r_k(y; \kappa_2 \Sigma) \},\\
		&= \{\Delta \in \bR^d  \mid \forall y \in \cY,  g_k (y, - \Delta) \leq \kappa_1 \sqrt{\kappa_2} \lambda  r_k(y; \Sigma)\\
		&= \{\Delta \in \bR^d  \mid \forall y \in \cY,  g_k (y, - \Delta / \kappa_1 \sqrt{\kappa_2}   ) \leq \lambda r_k(y; \Sigma)\\
		&= \kappa_1 \sqrt{\kappa_2}  S^+_k(\lambda, \cY; \nu \Sigma).
	\end{align*}
	Similarly, it holds that $S^-_k(\kappa_1 \lambda, \cY; \kappa_2 \Sigma) = \kappa_1 \sqrt{\kappa_2} S^-_k(\lambda, \cY; \Sigma)$.
	Since $S(\lambda, \cY; \Sigma) = \bigcap_{k=1}^K (S^+_k \cap S^-_k)$,
	these then imply that $S_k(\kappa_1 \lambda, \cY; \kappa_2 \Sigma) = \kappa_1 \sqrt{\kappa_2} S_k(\lambda, \cY; \Sigma)$.
	It then holds that
	\begin{align*}
		&\mu(p, \cY; \cN_{\nu_1 \Sigma}, \nu_2 \Sigma) \\
		&=  \inf \left\{ \mu \geq 0 \mid  \Prob_{\theta \sim \cN_{\nu_1 \Sigma}} \left(\theta \in  S(\mu, \cY; \nu_2 \Sigma) \right)   \geq p \right\} \\
		&=  \inf \left\{ \mu \geq 0 \mid  \Prob_{\theta \sim \cN_{\Sigma}} \left(\sqrt{\nu_1 } \theta \in \sqrt{\nu_2}  S(\mu, \cY; \Sigma) \right)   \geq p \right\} \\
		&=  \inf \left\{ \mu \geq 0 \mid  \Prob_{\theta \sim \cN_{\Sigma}} \left( \theta \in \sqrt{\nu_2 / \nu_1}  S(\mu, \cY; \Sigma) \right)   \geq p \right\} \\
		&=  \inf \left\{ \mu \geq 0 \mid  \Prob_{\theta \sim \cN_{\Sigma}} \left( \theta \in   S( \sqrt{\nu_2 / \nu_1} \mu, \cY; \Sigma) \right)   \geq p \right\} \\
		&= \sqrt{\nu_1 / \nu_2} \mu(p, \cY; \cN_{\Sigma}, \Sigma).
	\end{align*}
	The statement thus holds.
\end{proof}

\begin{proof}[Proof of Proposition~\ref{propNormalApprox}]
	By Lemma~\ref{lemParaEst}, it holds that
	\begin{align*}
		&\mu(p, \cY; \cP^*_n, \Sigma^*) \\
		&= \inf \left\{ \mu \geq 0 \left| \underset{\theta \sim \cP^*_n}{\Prob} \left(\theta \in  S(\mu, \cY; \Sigma) \right)   \geq p\right\}\right. \\
		&\leq \inf \left\{ \mu \geq 0 \left|  \underset{\theta \sim \cN_{\Sigma^*/n}}{\Prob} \left(\theta \in  S(\mu, \cY; \Sigma) \right)   \geq p + \frac{c_d \tau}{\sqrt{n}} \right\}\right. \\
		&= \mu \left( p  + \frac{c_d \tau}{\sqrt{n}} , \cY; \cN_{\Sigma^*/n}, \Sigma^* \right).
	\end{align*}
	Thus the second inequality holds. The first inequality follows from the same line of discussion. 
\end{proof}

\begin{proof}[Proof of Lemma~\ref{lemAlgo}]
	(i) For the left inequality, let up pick arbitrary $y \in \cY$, and then it holds that
	\begin{align*}
		\mu'(p) &= \inf \left\{ \mu \geq 0 \left| \underset{\theta \sim \cN_{\Sigma}}{\Prob} \left(\theta \in  S(\mu, \cY; \Sigma) \right)   \geq p \right\} \right. \\
		&\geq  \inf \left\{ \mu \geq 0 \left| \underset{\theta \sim \cN_{\Sigma}}{\Prob} \left(|g_1(y, \theta) | \leq  \mu r_1(y; \Sigma) \right)   \geq p \right\} \right. \\
		&= \chi_1^{-1}(p).
	\end{align*}
	For the right inequality, since $\mu U_{\Sigma} \subseteq  S(\mu, \cY; \Sigma)$ for any $\cY \subseteq \cX$ and $\Sigma$.
	It then holds that 
	\begin{align*}
		\mu'(p) &\leq \left\{ \mu \geq 0  \left| \underset{\theta \sim \cN_{\Sigma}}{\Prob} \left(\theta \in  \mu U_{\Sigma} \right)   \geq p \right\} \right. 
		= \chi_d^{-1}(p)
	\end{align*}
	
	(ii) This statement directly follows from the definition of $\mu(p, \cY; \cN_{\Sigma}, \Sigma)$ and $S(\lambda, \cY; \Sigma)$.
\end{proof}

\begin{proof}[Proof of Lemma~\ref{lemInS}]
	This statement directly follows from the definition of $S(\lambda, \cY; \Sigma)$.
\end{proof}

\begin{proof}[Proof of Proposition~\ref{propSampleAlgo}]
	Let us define $\dot{\mu}(p)$ by
	\begin{align*}
		\dot{\mu}(p) := \inf \left\{ \mu \geq 0 \left|  \ | \{ i \mid \psi(\mu, \varepsilon_i) \geq 0  \} | \geq p Q \right\} \right. .
	\end{align*}
	Then the output $\overline{\lambda}$ satisfies 
	\begin{align}
		\dot{\mu}(p + \beta/2) \leq  \overline{\lambda} \leq \dot{\mu}(p + \beta/2) + \gamma. \label{ineqPropSample}
	\end{align}
	By Lemma~\ref{lemHoeffding}, then, the following inequalities holds with probaiblity at least $1 - \alpha/2$:
	\begin{align*}
		\underset{\theta \sim \cN_{\Sigma}}{\Prob} \left(\theta \in  S(\dot{\mu}(p + \beta/2), \cY; \Sigma) \right) \geq (p + \beta/2) - \beta/2 = p,
	\end{align*}
	which shows that $\mu'(p) \leq \dot{\mu}(p + \beta/2)$.
	Similarly, with probability at least $1 - \alpha/2$, it holds that $\dot{\mu}(p + \beta/2) \leq \mu'(p + \beta)$.
	By \eqref{ineqPropSample} and the union bound, then, the following holds with probability at least $1 - \alpha$:
	\begin{align*}
		\mu'(p) \leq  \dot{\mu} \leq \mu'(p + \beta) + \gamma.
	\end{align*}
	The statement thus holds.
\end{proof}

\subsection{Proofs for the statements in Section~\ref{secReduction}}
\begin{proof}[Proof of Theorem~\ref{thmFormal}]
	Let $\varepsilon_1 = d^2 c_1 \tau' / \sqrt{n}$.
	By Lemma~\ref{lemCovEst}, with probability at least $1 -  2 \varepsilon_1$,
	it holds that 
	\begin{align}
		\left( 1 - \eta_n  \right) \Sigma^*    \preceq \hat{\Sigma}_n  \preceq \left( 1 + \eta_n \right) \Sigma^* . \label{ineqThmSigma}
	\end{align}
	
	Suppose that \eqref{ineqThmSigma} holds.
	Let us define $\varepsilon_2 := c_d \tau /\sqrt{n}$.
	By Proposition~\ref{propSampleAlgo},
	$\hat{\mu}$ satisfy the following with probability at least $1 - \delta/ \sqrt{n}$:
	\begin{align}
		\mu(1 - 2\varepsilon_2, \cX, \cN_{\hat{\Sigma}_n}, \hat{\Sigma}_n)
		\leq \dot{\mu}_1
		\leq \mu(1 - \varepsilon_2, \cX, \cN_{\hat{\Sigma}_n }, \hat{\Sigma}_n ) + \frac{1}{\sqrt{n}} \label{ineqThmDotMu}
	\end{align}
	Then it holds that
	\begin{align}
		&\underline{\mu}_n := \mu(1 - 3\varepsilon_2, \cX, \cP^*_n, \Sigma^*) \nonumber \\
		&\leq \underline{\mu}_n' := \frac{1}{\sqrt{n}} \mu(1 - 2\varepsilon_2, \cX, \cN_{\Sigma^*}, \Sigma^*)\nonumber  \\
		&\leq \frac{1}{\sqrt{n}} \mu\left(1 - 2\varepsilon_2, \cX, \cN_{\hat{\Sigma}_n / (1 - \eta_n)}, \hat{\Sigma}_n / (1 + \eta_n) \right)\nonumber  \\
		&= \frac{1 + \eta_n}{ (1 - \eta_n)\sqrt{n}}  \mu(1 - 2\varepsilon_2, \cX, \cN_{\hat{\Sigma}_n}, \hat{\Sigma}_n )\nonumber  \\
		&\leq  \frac{1 + \eta_n}{ (1 - \eta_n)\sqrt{n}} \dot{\mu}_1   \nonumber  \\ 
		&=  \hat{\mu}\nonumber \\
		&\leq \overline{\mu}_n := \frac{(1 + \eta_n)^2}{( 1 - \eta_n)^2 \sqrt{n}} \mu(1 - \varepsilon_2, \cX, \cN_{\Sigma^*}, \Sigma^*)  + \frac{1 + \eta_n}{ (1 - \eta_n) n } \nonumber \\
		&\leq \frac{(1 + \eta_n)^2}{( 1 - \eta_n)^2 \sqrt{n}}  \chi_{d}^{-1} \left(1 - \varepsilon_2 \right)  + \frac{1 + \eta_n}{ (1 - \eta_n) n } . \label{inequThmMu}
	\end{align}
	The first inequality follows from Proposition~\ref{propNormalApprox}.
	The second inequality follows from \eqref{ineqThmSigma} and Lemma~\ref{lemMonotonicity} (iii).
	The first equality follows from Lemma~\ref{lemMonotonicity} (iv).
	The third inequality follows from \eqref{ineqThmDotMu}.
	The forth inequality follows from \eqref{ineqThmSigma}, Lemma~\ref{lemMonotonicity} (iii) and (iv), and \eqref{ineqThmDotMu}.
	The fifth inequality follows from Lemma~\ref{lemAlgo} (i).
	
	Suppose that \eqref{ineqThmSigma} and \eqref{inequThmMu} hold.
	By the definition of the minimum spatial uniform bound,
	the following holds with probability at least $1 - 3\varepsilon_2 $:
	\begin{align}
		\hat{\theta}_n - \theta^* \in S(\underline{\mu}_n, \cX; \Sigma^*) &\subseteq S(\underline{\mu}_n', \cX; \Sigma^*)\\
		&\subseteq S(\hat{\mu}, \cX; \hat{\Sigma}_n / (1 - \eta_n)). \label{ineqThmHatTheta1}
	\end{align}
	Observe that $\hat{x}(3 \hat{\mu} /(1 - \eta_n)  ) = x( 3 \hat{\mu}, \hat{\theta}_n, \hat{\Sigma}_n / (1 - \eta_n) )$.
	By Lemma~\ref{lemConf2} with $\cY = \cX$, \eqref{ineqThmSigma}, and \eqref{inequThmMu}, it holds that
	\begin{align*}
		\underline{w}_n &:= f\left( x^*(2\underline{\mu}_n') \right) \\
		&\leq f\left( x(3\underline{\mu}_n'; \hat{\theta}_n, \Sigma^*) \right) \\
		&\leq \hat{w} = f\left(x( 3 \hat{\mu}, \hat{\theta}_n, \hat{\Sigma}_n / (1 - \eta_n) ) \right)  \\
		&\leq f\left(x( 4\hat{\mu}, \theta^*, \hat{\Sigma}_n / (1 - \eta_n) ) \right)  \\
		&\leq f\left(x( 4\hat{\mu}, \theta^*, (1 + \eta_n) \Sigma^* / (1 - \eta_n) ) \right)  \\
		&\leq \overline{w}_n :=  f \left(x^* \left( \frac{4 (1 + \eta_n)  \overline{\mu}}{ (1 - \eta_n) } \right) \right) .
	\end{align*}
	Observe that, by \eqref{ineqThmSigma} and \eqref{ineqThmHatTheta1}, for any $w'$ and $\mu' \geq 0$, it holds that
	\begin{align*}
		\cY^*(w', \mu') &\subseteq \hat{\cY}\left(w', \frac{1}{1 - \eta_n} (\mu' + \underline{\mu}_n) \right) \\
		&\subseteq \cY^*\left(w', \frac{1+ \eta_n}{1 - \eta_n} (\mu' + 2\underline{\mu}_n) \right).
	\end{align*}
	It then holds that
	\begin{align*}
		&\underline{\cY}_n := \cY^*(\underline{w}_n, 0) \\
		&\subseteq \hat{\cY} = \hat{\cY}\left(\hat{w}, \frac{1}{1 - \eta_n}  \hat{\mu} \right) \\
		&\subseteq \overline{\cY}_n := \cY^*\left(\overline{w}_n, \frac{2(1+ \eta_n)}{1 - \eta_n}\overline{\mu}_n \right).
	\end{align*}
	
	Suppose that \eqref{ineqThmSigma}, \eqref{inequThmMu} and \eqref{ineqThmHatTheta1} holds.
	Recall that $\delta' = 2 \varepsilon_1 +  \delta/\sqrt{n} + 4\varepsilon_2$.
	By \eqref{ineqThmSigma}, Proposition~\ref{propNormalApprox}, and Lemma~\ref{lemAlgo} (i), 
	it holds that
	\begin{align}
		&\underline{\lambda}_n :=  \mu(1 - \delta + \delta' - \varepsilon_2, \underline{\cY}_n, \cP^*_n, \Sigma^*)\\
		&\leq  \frac{1}{\sqrt{n}} \mu(1 - \delta + \delta', \hat{\cY}, \cN_{\Sigma^*}, \Sigma^*)\nonumber  \\
		&\leq  \frac{1 + \eta_n}{(1 - \eta_n)\sqrt{n}} \mu(1 - \delta + \delta', \hat{\cY}, \cN_{\hat{\Sigma}_n}, \hat{\Sigma})  \nonumber\\
		&\leq  \frac{1 + \eta_n}{(1 - \eta_n)\sqrt{n}} \dot{\lambda} \nonumber\\
		&= \hat{\lambda} \nonumber \\
		&\leq \overline{\lambda}_n := \frac{(1 + \eta_n)^2}{( 1 - \eta_n)^2 \sqrt{n}} \mu(1 - \delta + \delta' + \varepsilon_2, \overline{\cY}_n, \cN_{\Sigma^*}, \Sigma^*)  \nonumber \\
		&\qquad \qquad + \frac{1 + \eta_n}{ (1 - \eta_n) n }. \label{inequThmLmd}
	\end{align}
	By the definition of the minimum spatial uniform bound,
	with probability at least $1 - \delta + \delta'$, it holds that
	\begin{align}
		\hat{\theta}_n - \theta^* \in S(\underline{\lambda}_n, \underline{\cY}_n; \Sigma^*). \label{ineqThmHatTheta2}
	\end{align}
	
	By the union bound, 
	\eqref{ineqThmSigma}, \eqref{inequThmMu}, \eqref{ineqThmHatTheta1}, and \eqref{ineqThmHatTheta2} hold with probability at least $1-\delta$.
	Since $\hat{\lambda} \geq 0$, it is immediate that
	$\hat{x}(\hat{\lambda}) \in \hat{\cY}(\infty, \hat{\mu})$.
	Since $n \geq n_0$, the following inequalities hold.
	\begin{align*}
		\hat{\lambda} &\leq \frac{(1 + \eta_n)^2}{( 1 - \eta_n)^2 \sqrt{n}} \mu(1 - \delta + \delta' + \varepsilon_2, \overline{\cY}_n, \cN_{\Sigma^*}, \Sigma^*)  + \frac{1 + \eta_n}{ (1 - \eta_n) n } \\
		&\leq \frac{(1 + \eta_n)^2}{( 1 - \eta_n)^2 \sqrt{n}} \mu(1 - \delta + \delta' + \varepsilon_2, \cX, \cN_{\Sigma^*}, \Sigma^*)  + \frac{1 + \eta_n}{ (1 - \eta_n) n } \\
		&\leq \frac{1}{\sqrt{n}} \mu(1 - 2\varepsilon_2, \cX, \cN_{\Sigma^*}, \Sigma^*)\\
		&\leq \underline{\mu}_n'
	\end{align*}
	It then holds that
	\begin{align*}
		f(\hat{x}(\hat{\lambda})) \leq \underline{w}_n.
	\end{align*}
	Therefore it holds that $\hat{x}(\hat{\lambda}) \in \underline{\cY}_n$.
	Also, since $2 \hat{\lambda} \leq 2 \underline{\mu}_n'$ imply
	$f(x^*(2 \hat{\lambda}) ) \leq f(x^*(2\underline{\mu}_n' ))$
	\begin{align}
		f^*(\hat{x}(\hat{\lambda}) ) \leq f(x(2\hat{\lambda}; \theta^*, \hat{\Sigma}_n)) 
		\leq f^*(x^*(2\hat{\lambda} / (1 - \eta_n))) \leq \underline{w}_n, \label{ineqThmFinal}
	\end{align}
	and thus $\hat{x}(\hat{\lambda}) \in \underline{\cY}_n$.
	Therefore, by Lemma~\ref{lemConf}, it holds that
	\begin{align*}
		f^*(\hat{x}(\hat{\lambda})) \leq f^*(x(2 \hat{\lambda}; \theta^* , \hat{\Sigma}_n )) 
		\leq f^*(x^*(2 \hat{\lambda}/ (1 - \eta_n))) \leq f^*(x^*(2 \overline{\lambda}_n/ (1 - \eta_n)))
	\end{align*}
	In summary, with probability $1-\delta$, \eqref{ineqThmFinal} holds.
	
	Observe that $\lim_{n \to \infty} \overline{\cY}_n = \{x^*(0)\}$ by the definition of $\overline{\cY}_n$ 
	and the assumption $\lim_{\varepsilon \to +0} \cY^*(\varepsilon, \varepsilon) = \{x^*(0) \}$.
	By Lemma~\ref{lemMonotonicity} (i) and (ii), it holds that
	\begin{align*}
		\lim_{n \to \infty } \sqrt{n} \overline{\lambda}_n 
		&=\lim_{n \to \infty} \mu(1 - \delta+ \delta'_n, \overline{\cY}_n ; \cN_{\Sigma^*}, \Sigma^*)\\
		&=\mu(1- \delta, \{x^*(0) \}; \cN_{\Sigma^*}, \Sigma^*)
		\leq \chi_{1}^{-1}(1 - \delta/K).
	\end{align*}
	The last inequality follows from the union bound,
	and thus the proof is complete.
\end{proof}

\begin{proof}[Proof of Lemma~\ref{lemConf2}]
	By the definition of $r_k$, $S(\lambda, \cY; \Sigma)$, and linearlity of $g_k$ with respect to $\theta$, for all $y \in \cY$ and $k=1,2,\dots,K$, it holds that
	\begin{align*}
		&\min_{u \in U_{\Sigma}} g_k (y, \theta^* + (\kappa+1)\lambda u ) = g_k(y, \theta^*) - (\kappa+1)\lambda r_k (y; \Sigma) \\
		&= g_k(y, \theta) + g_k(y, \theta^* - \theta) - (\kappa+1)\lambda r_k (y; \Sigma) \\
		&\leq g_k(y, \theta) - \kappa \lambda r_k (y; \Sigma) = \min_{u \in U_{\Sigma}} g_k (y, \theta^* + \kappa\lambda u ) \\
		&= g_k(y, \theta^*) + g_k(y, \theta - \theta^*)  - \kappa\lambda r_k (y; \Sigma)\\
		&\leq \min_{u \in U_{\Sigma}} g_k(y, \theta^* + (\kappa-1)\lambda u) \leq  g_k(y, \theta^* ).
	\end{align*}
	Since $x((\kappa-1)\lambda; \theta^*, \Sigma), x(\kappa \lambda; \theta, \Sigma), x((\kappa + 1)\lambda; \theta^*, \Sigma) \in \cY$,
	the inequalities above indicates that $f^*(x((\kappa-1)\lambda; \theta^*, \Sigma)) \leq f^*(x(\kappa\lambda; \theta,\Sigma)) \leq f^*(x((\kappa+1)\lambda; \theta^*,\Sigma))$.
\end{proof}

\end{document}